\documentclass{amsart}
\usepackage{amsmath}
\usepackage{amssymb}
\usepackage{amsthm}
\usepackage{amscd}
\usepackage{comment}
\usepackage{url}
\usepackage{pst-node}
\usepackage{auto-pst-pdf}
\usepackage{tikz-cd} 
\usepackage{tikz}
\usetikzlibrary{arrows.meta,calc,chains,positioning,quotes,shapes,shapes.multipart}
\usetikzlibrary{shapes,arrows}

\tikzstyle{myarrow}=[->, >=open triangle 90, thick]
\tikzstyle{line} = [draw, -latex']
 \tikzstyle{Vertex} = [draw, circle, minimum width = .25cm, text width=.25cm,  text centered]
\tikzstyle{box} = [draw, rectangle, minimum width = 1.5cm, text width=.25cm,  text centered]

\newtheorem{theorem}{Theorem}[section]

\newtheorem{lemma}[theorem]{Lemma}
\newtheorem{corollary}[theorem]{Corollary}
\newtheorem{proposition}[theorem]{Proposition}

\theoremstyle{definition}
\newtheorem{definition}[theorem]{Definition}
\newtheorem{example}[theorem]{Example}

\theoremstyle{remark}
\newtheorem{remark}[theorem]{Remark}

\numberwithin{equation}{section}

\begin{document}
	
	\title[A Radon-Nikod\'ym theorem for CP-maps on Hilbert pro-$C^*$-modules]{A Radon-Nikod\'ym theorem for completely positive maps on Hilbert pro-$C^*$-modules}


 \author{Bhumi Amin}
	\address{Department of Mathematics, \\IIT Hyderabad, \\Telangana, India - 502285 }
\email{ma20resch11008@iith.ac.in}

	\author{Ramesh Golla}
	\address{Department of Mathematics,\\ IIT Hyderabad, \\Telangana, India - 502285}

	\email{rameshg@math.iith.ac.in}

	
	
	
	
	
		\begin{abstract}
			We introduce an equivalence relation on the set of all completely positive maps between Hilbert modules over pro-$C^*$-algebras and analyze the Stinespring's construction for equivalent completely positive maps. We then give a pre-order relation in the collection of all completely positive maps between Hilbert modules over pro-$C^*$-algebras and obtain a Radon-Nikod\'ym type theorem.
		\end{abstract}
	
	\maketitle
	
	

\section{Introduction}
The study of completely positive maps (CP-maps) is driven by their significant applications in modern mathematics, including quantum information theory, statistical physics, and stochastic processes. In particular, operator-valued CP-maps on $C^*$-algebras are used to model quantum operations and quantum probabilities (see \cite{VP} for more details on CP-maps).
Stinespring \cite[Theorem 1]{WS} established that any operator-valued completely positive map \(\phi\) on a unital \(C^*\)-algebra \(\mathcal{A}\) can be represented in the form  
$
{V_\phi}^* \pi_\phi(\cdot) V_\phi,
$ 
where \(\pi_\phi\) is a representation of \(\mathcal{A}\) on a Hilbert space \(H\), and \(V_\phi\) is a bounded linear operator.

The Radon-Nikod\'ym theorem, a fundamental result in measure theory, expresses the relationship between two measures defined on the same measurable space. 
The theorem was subsequently generalized to $W^*$-algebras, von-Neumann algebras, and $^*$-algebras, in that order (see references \cite{SS, PT, SG}). In 1983, Atsushi Inoue introduced a Radon-Nikod\'ym theorem for positive linear functionals on $^*$-algebras in \cite{AI2}. Additionally, a Radon-Nikod\'ym theorem for completely positive maps was developed by Belavkin and Staszewski in 1986 (see \cite{BS} for more details).

Given two operator valued completely positive maps $\phi$ and $\psi$ on a $C^*$-algebra $\mathcal{A}$, a natural partial order is defined by $\phi \leq \psi$ if $\psi - \phi$ is completely positive. Arveson, in \cite{AW}, characterized this relation using the Stinespring construction associated with each completely positive map and introduced the notion of the Radon-Nikod\'ym derivative for operator-valued completely positive maps on $C^*$-algebras. He proved that $\phi \leq \psi$ if and only if there exists a unique positive contraction $\Delta_\phi(\psi)$ in the commutant of $\pi_\phi(\mathcal{A})$ such that $\psi(.) = {V_\phi}^*\Delta_\phi(\psi)\pi_\phi(.)V_\phi$.

Hilbert modules over \(C^*\)-algebras serve as a natural generalization of Hilbert spaces, where the inner product takes values in a \(C^*\)-algebra rather than the complex numbers.
Kaplansky first introduced Hilbert modules over unital, commutative \(C^*\)-algebras in \cite{KI}. Paschke in 1973, and Rieffel in 1974, later extended the study to Hilbert modules over arbitrary \(C^*\)-algebras (refer to \cite{Pash2,RM}).  
A Stinespring-type representation for operator-valued completely positive maps on Hilbert modules over \(C^*\)-algebras is given by Asadi in \cite{AMB}.
 A refinement of this result was given by Bhat, Ramesh, and Sumesh in \cite{BRS}. Building on \cite[Theorem 2.1]{BRS}, Skeide developed a factorization theorem in \cite{MS2} using induced representations of Hilbert modules over $C^*$-algebras. In \cite{BR}, a Stinespring-like theorem for maps between two Hilbert modules over respective pro-$C^*$-algebras is established. 
We primarily utilized this result, along with additional definitions from \cite{BR}, to prove our results.

The concept of locally $C^*$-algebras, which is a generalization of $C^*$-algebras, was introduced by A. Inoue in 1971 (see \cite{AI} for more details). A locally \(C^*\)-algebra is a complete topological involutive algebra, where the topology is determined by a family of seminorms. These algebras are also referred to as ``pro-\(C^*\)-algebras," a term that will be used throughout this paper.
Rieffel defined induced representations of $C^*$-algebras (see \cite{RM}). Induced representations of Hilbert modules over pro-\( C^* \)-algebras have been studied in depth in \cite{MJ2, KS1}.
 In 1988, Phillips \cite{NCP} characterized a topological $^*-$algebra $\mathcal{A}$ as a pro-$C^*$-algebra if it is the inverse limit of an inverse system of $C^*$-algebras and $^*-$homomorphisms. Using this setup, Hilbert modules over a pro-$C^*$-algebra can be defined, which we refer to as Hilbert pro-$C^*$-modules.

Joiţa \cite{MJ}, in 2012, established a preorder relation for operator-valued completely positive maps on a Hilbert module over $C^*$-algebras and established a Radon-Nikod\'ym-type theorem for these maps.
 In 2017, Karimi and Sharifi \cite{KS} presented a Radon-Nikod\'ym theorem for operator valued completely positive maps on Hilbert modules over pro-$C^*$-algebras. These contributions form the primary motivation for our research. In this paper, we establish an equivalence relation on the set of all completely positive maps between two Hilbert pro-$C^*$-modules, demonstrating that the Stinespring constructions for equivalent completely positive maps are equivalent in some sense. Additionally, we introduce a preorder relation for completely positive maps between two Hilbert pro-$C^*$-modules and prove a Radon-Nikod\'ym-type theorem for these maps.

\section{Preliminaries}

Throughout this paper, we focus on algebras over the complex field. First, let's review the definitions of pro-$C^*$-algebras and Hilbert modules over these algebras.

\begin{definition}\cite[Definition 2.1]{AI}
    A $^*-$algebra $\mathcal{A}$ is called a pro-$C^*$-algebra if there exists a family $\{p_j\}_{j\in J}$ of semi-norms defined on $\mathcal{A}$ such that the following conditions are satisfied: 
    \begin{enumerate}
        \item $\{p_j\}_{j\in J}$ defines a complete Hausdorff locally convex topology on $\mathcal{A}$;
        \item $p_j(x y) \leq p_j(x) p_j(y),$ for all $x, y \in \mathcal{A}$ and each $j \in J$;
        \item $p_j(x^*) = p_j(x),$ for all $x \in \mathcal{A}$  and each $j \in J$;
        \item $p_j(x^*x) = p_j(x)^2,$  for all $x \in \mathcal{A}$  and each $j \in J$.
    \end{enumerate}
\end{definition}
We call the family $\{p_j\}_{j\in J}$ of semi-norms defined on $\mathcal{A}$ as the family of $C^*$-semi-norms. Let $S(\mathcal{A})$ denote the set of all continuous $C^*$-semi-norms on $\mathcal{A}.$ If the pro-$C^*$-algebra $\mathcal{A}$ is unital, we denote its unit by $1_\mathcal{A}$. The following are few examples of a pro-$C^*$-algebra:

\begin{enumerate}
    \item Consider the set $\mathcal{A} = C(\mathbb{R}),$ the set of all continuous complex valued functions on $\mathbb{R}.$  Then $\mathcal{A}$ forms a pro-$C^*$-algebra, with the locally convex  Hausdorff topology induced by the family $\{p_n\}_{n \in \mathbb{N}}$ of seminorms given by, 
$$p_n(f) = \text{sup}\{|f(t)|: t \in [-n,n]\}.$$

     \item        A product of $C^*$-algebras with product topology is a pro-$C^*$-algebra. Indeed, for a collection $\{A_j\}_{j \in J}$ of $C^*$-algebras, define
$$p_j((a_j)_j) : = \|a_j\|_j,$$
where $\|\cdot\|_j$ denotes the $C^*$-norm on $A_j.$

\end{enumerate}


Consider two pro-$C^*$-algebras $\mathcal{A}$ and $\mathcal{B}$. 

An element $a \in \mathcal{A}$ is called positive (denoted by $a \geq 0$), if there is an element $b \in \mathcal{A}$ such that $a = b^*b$.
A linear map $\phi: \mathcal{A} \rightarrow \mathcal{B}$ is said to be positive if  $\phi(a^*a)\geq 0,$ for all $a \in \mathcal{A}$. By $M_n(\mathcal{A})$ we denote the set all of $n \times n$ matrices with entries from $\mathcal{A}.$ Note that $M_n(\mathcal{A})$ is a pro-$C^*$-algebra (see \cite{MJbook} for futher details).

	\begin{definition}
		 A linear map $\phi: \mathcal{A} \rightarrow \mathcal{B}$ is said to be completely positive\index{completely positive} (or CP), if for all $n \in \mathbb{N}, \phi^{(n)}: M_n(\mathcal{A}) \rightarrow M_n(\mathcal{B})$ defined by $$\phi^{(n)}([a_{ij}]_{i,j=1}^n) = [\phi(a_{ij})]_{i,j=1}^n$$ is positive.
	\end{definition}

Note that, for completely positive maps $\phi_1, \phi_2: \mathcal{A} \rightarrow \mathcal{B}$, we write $\phi_1 \leq \phi_2$, whenever $\phi_2 - \phi_1$ is completely positive.

	\begin{definition}\cite[Definition 1.1.6]{MJbook}
A $^*-$morphism $\phi: \mathcal{A} \rightarrow \mathcal{B}$ is a linear map such that the following two conditions are satisfied:
		\begin{itemize}
			\item[(1)] $\phi(ab) = \phi(a)\phi(b),$ for all $a, b \in \mathcal{A}$; 
			\item[(2)] $\phi(a^*) = \phi(a)^*,$ for all $a \in \mathcal{A}.$
		\end{itemize}
	\end{definition}

\begin{definition}
    Let 
    $E$ a complex vector space that is also a right $\mathcal{A}$-module. We call $E$ to be a pre-Hilbert $\mathcal{A}$-module if it has an $\mathcal{A}$-valued inner product $\langle \cdot, \cdot \rangle: E \times E \rightarrow \mathcal{A}$, which is $\mathbb{C}$-linear and $\mathcal{A}$-linear in the second variable, and meets the following conditions:
\begin{enumerate}
    \item $\langle \xi, \eta \rangle^* = \langle \eta, \xi \rangle$ for all $\xi, \eta \in E$;
    \item $\langle \xi, \xi \rangle \geq 0$ for all $\xi \in E$;
    \item $\langle \xi, \xi \rangle = 0$ if and only if $\xi = 0$.
\end{enumerate}

We say that $E$ is a Hilbert $\mathcal{A}$-module if it is complete with respect to the topology defined by the $C^*-$seminorms $\{\| \cdot \|_p\}_{p \in S(\mathcal{A})}$, 
$$\|\xi\|_p := \sqrt{p(\langle \xi, \xi \rangle)}, \ \text{  for $\xi \in E$.}$$

\end{definition}
When working with multiple Hilbert modules over the same pro-$C^*$-algebra, we use the notation $\| \cdot \|_{p_E}$ instead of $\| \cdot \|_p$.

\begin{definition}
A closed submodule $E_0$ of a Hilbert $\mathcal{A}$-module $E$ is said to be complemented if 
\[
E = E_0 \oplus E_0^\perp,
\]
where $
E_0^\perp = \{x \in E : \langle x, y \rangle = 0 \text{ for every } y \in E_0\}.
$
\end{definition}
We use $E \ominus E_0$ to denote the orthogonal complement ${E_0}^\perp$ of $E_0$ in $E$.

For our results, we'll employ a modified version of the following well-known definition, referencing it as needed as seen in \cite{BR}.
 	\begin{definition}
        Let $E$ be a Hilbert $\mathcal{A}$-module and $F$ be a Hilbert $\mathcal{B}$-module.

		Let $\phi: \mathcal{A} \rightarrow \mathcal{B}$ be a linear map. A map $\Phi: E \rightarrow F$ is said to be 
		\begin{enumerate}
			\item a $\phi-$map, if $$\langle \Phi(x), \Phi(y) \rangle = \phi(\langle x , y \rangle), $$ for all $x, y \in E$;
                \item continuous, if $\Phi$ is a $\phi-$map and $\phi$ is continuous;
			\item a $\phi-$morphism, if $\Phi$ is a $\phi-$map and $\phi$ is a $^*-$morphism;
			\item  completely positive, if $\Phi$ is a $\phi-$map and $\phi$ is completely positive.
		\end{enumerate}
	\end{definition}

	 The set $\langle E, E \rangle$ denotes the closure of the linear span of $\{\langle x,y\rangle: x,y \in E\}.$ If $\langle E, E \rangle = \mathcal{A}$, then $E $ is said to be a full Hilbert module.

Next, let $E$ and $F$ be Hilbert modules over a pro-$C^*$-algebra $\mathcal{B}.$  A map $T: E \rightarrow F$ is said to be a $\mathcal{B}$-module map if $T$ is $\mathcal{B}-$linear, that is, 
  $$ T(e_1 + e_2) = T(e_1) + T(e_2) \text{ and } T(e b) = T(e)b,$$
  for all $e, e_1, e_2 \in E$ and $b \in \mathcal{B}$.

\begin{definition}\cite[Definition 2.1.1]{MJbook}
A $\mathbb{C}$- and $\mathcal{B}$-linear map $T: E \rightarrow F$ is called a bounded operator\index{bounded operator} if, for each $p \in S(\mathcal{B})$, there exists a constant $M_p > 0$ such that  
\[
\|T(x)\|_{p_F} \leq M_p \|x\|_{p_E},
\]  
for all $x \in E$.  
\end{definition}

A bounded $\mathcal{B}$-linear map $T: E \rightarrow F$ is said to be adjointable if there exists a map $T^*: F \rightarrow E$ such that, for all $\xi\in E$ and $\eta\in F$, the following condition holds: \begin{equation*}
		\langle T\xi, \eta \rangle = \langle \xi, T^*\eta \rangle.
	\end{equation*}
   By $\mathcal{L}_\mathcal{B}(E,F),$ we denote the set of all bounded adjointable $\mathcal{B}$-module operators from $E$ to $F$ with inner-product defined by 
   \begin{equation*}
   \langle T, S \rangle := T^*S,\; \text{ for}\; T, S \in \mathcal{L}_\mathcal{B}(E,F).
   \end{equation*}
      Note that $\mathcal{L}_\mathcal{B}(E,F)$ is a Hilbert $\mathcal{L}_\mathcal{B}(E)$-module with the module action 

   \begin{equation*}
   (T,S) \rightarrow TS, \; \text{ for} \; T \in \mathcal{L}_\mathcal{B}(E,F)\; \text{ and} \;S \in {L}_\mathcal{B}(E).
   \end{equation*}
   We denote the set $\mathcal{L}_\mathcal{B}(E,E)$ by $\mathcal{L}_\mathcal{B}(E).$ The set $\mathcal{L}_\mathcal{B}(E)$ in fact forms a pro-$C^*$-algebra (see \cite[Theorem 2.2.6]{MJbook}).

   An operator $T \in \mathcal{L_B}(E)$ is said to be positive, denoted by $T \geq 0$, if   $\langle Tx, x\rangle \geq 0$ for all $x \in E.$ For $T, S \in \mathcal{L}_{\mathcal{B}}(E)$, we write $T \leq S$ ( or $S \geq T$) to mean that $S - T$ is positive. Note that $T$ is positive if and only if there exists $S \in \mathcal{L}_{\mathcal{B}}(E)$ such that $T = S^*S$; this equivalence follows from the functional calculus.

      An operator $T \in \mathcal{L_B}(E)$ is called a projection if $P^*=P^2 = P.$

\begin{definition}\cite{MJbook}\label{leftaction}
A Hilbert $\mathcal{B}$-module $E$ is called a Hilbert $\mathcal{A}\mathcal{B}$-module if there exists a non-degenerate $^*-$homomorphism $\tau:\mathcal{A} \rightarrow \mathcal{L}_\mathcal{B}(E)$.

    In this case, we identify $a.e $ with $\tau(a).e$ for all $a \in \mathcal{A}$ and $e \in E.$
\end{definition}

By a Hilbert $\mathcal{B}$-module, we refer to a Hilbert (right) $\mathcal{B}$-module for any pro-$C^*$-algebra $\mathcal{B}$. A two-sided Hilbert $\mathcal{B}$-module is a Hilbert $\mathcal{BB}$-module.

We now state Paschke's GNS construction for completely positive maps on  pro-$C^*$-algebras (see \cite{KS} for more details).  
	
	\begin{theorem}\cite[Theorem 1]{KS}\label{Paschpro} 
		Let 
        $\phi:\mathcal{A} \rightarrow \mathcal{B}$ be a continuous completely positive map. Then there exists a Hilbert $\mathcal{B}$-module $X,$ a unital continuous representation $\pi_\phi: \mathcal{A} \rightarrow \mathcal{L}_\mathcal{B}(X),$ and an element $\xi \in X$ such that 
		$$\phi(a) = \langle \xi, \pi_\phi(a) \xi \rangle,$$
		for all $a \in \mathcal{A}$. Moreover, the set $\chi_\phi = \text{span}\{\pi_\phi(a)(\xi b): a \in \mathcal{A}, b \in \mathcal{B}\}$ is a dense subspace of $X.$
	\end{theorem}

The following theorem from \cite{BR} provides a Stinespring-like construction for completely positive maps between two Hilbert pro-$C^*$-modules. We will extensively use this construction for our results. The Hilbert \(\mathcal{B}\)-module \(X\) in this context is the one obtained from Paschke's GNS construction (Theorem~\ref{Paschpro}).

\begin{theorem}\cite[Theorem 3.9]{BR}\label{BRmain}
Let $\phi: \mathcal{A} \rightarrow \mathcal{B}$ be a continuous completely positive map between unital pro-$C^*$-algebras $\mathcal{A}$ and $\mathcal{B}$. Let $E$ be a Hilbert $\mathcal{A}$-module, $F$ be a Hilbert $\mathcal{BB}$-module and $\Phi: E \rightarrow F$ be a $\phi-$map. Then there exists a Hilbert $\mathcal{B}$-module $D$, a vector $\xi \in X$, and triples $\left(\pi_\phi,V_\phi,K_\phi\right)$ and $\left(\pi_\Phi,W_\Phi,K_\Phi\right)$ such that 
		\begin{enumerate}
			
			\item $K_\phi$ and $K_\Phi$ are  Hilbert $\mathcal{B}$-modules;
			\item$\pi_\phi: \mathcal{A} \rightarrow \mathcal{L}_\mathcal{B}(K_\phi)$ is a unital representation of $\mathcal{A}$;
			\item $\pi_\Phi: E \rightarrow \mathcal{L}_\mathcal{B}(K_\phi,K_\Phi)$ is a $\pi_\phi-$morphism;
			\item $V_\phi: D \rightarrow K_\phi$ and $W_\Phi: F \rightarrow K_\Phi$ are bounded linear operators such that 
			$$\phi(a)I_D = {V_\phi}^*\pi_\phi(a)V_\phi,$$ for all $a \in \mathcal{A}$, and
			$$\Phi(z) = {W_\Phi}^*\pi_\Phi(z)V_\phi (1_\mathcal{B} \otimes \xi),$$ for all $z \in E.$
		\end{enumerate}
\end{theorem}

Note that here $D$ is the Hilbert pro-$C^*$-module associated with $B \otimes X.$



\section{Main Results}

Let $\mathcal{A}$ and $\mathcal{B}$ be unital pro-$C^*$-algebras. Let $E$ be a Hilbert $\mathcal{A}$-module and $F$ be a Hilbert $\mathcal{BB}$-module. 

Define the set
$$\mathcal{CP}(E,F)= \{\Phi: E \rightarrow F: \Phi \text{ is continuous, completely positive}\}.$$

We would like to recall that  each $\Phi \in \mathcal{CP}(E,F),$ has a Stinespring's construction $(\pi_\Phi, K_\phi, K_\Phi, V_\phi, W_\Phi)$ attached to it by Theorem \ref{BRmain}.

From now onwards, apart from the mentioned notations, we assume $E$ to be full.

\begin{definition}[Equivalence relation]
    For $\Phi, \Psi \in \mathcal{CP}(E,F),$ we define the relation ``$\sim$" as follows: $$\Phi \sim \Psi \Leftrightarrow \langle \Phi(x), \Phi(x) \rangle = \langle \Psi(x), \Psi(x) \rangle,$$ for all $x \in E.$
\end{definition}
 It is easy to observe that $``\sim"$ is an equivalence relation.

 \begin{remark}
  Let $\Phi_1, \Phi_2 \in \mathcal{CP}(E,F)$ such that  $\Phi_1 \sim \Phi_2$. 
  \begin{enumerate}
      \item  For $x \in E,$ we have 
    \begin{equation*}
        \begin{split}
             \phi_1\left(\langle x, x \rangle\right) 
             &= \langle \Phi_1(x), \Phi_1(x) \rangle \\ 
             &= \langle \Phi_2(x), \Phi_2(x) \rangle \\
             &= \phi_2\left(\langle x, x \rangle\right).
        \end{split}
    \end{equation*} Since $E$ is full, by polarization, we get $\phi_1 = \phi_2$.
   
    \item Let $G$ be a two-sided Hilbert pro-$C^*$-module over a pro-$C^*$-algebra $\mathcal{C}.$ Let $\Theta \in \mathcal{CP}(F,G)$. Then $\Theta \circ \Phi_1 \sim \Theta \circ \Phi_2$. Indeed, for $x \in E,$ we have 
    \begin{equation*}
        \begin{split}
\left \langle \Theta \circ \Phi_1(x) , \Theta \circ \Phi_1(x)\right \rangle 
             &= \theta\circ\phi_1(\langle x,x\rangle) \\
               &= \theta\circ\phi_2(\langle x,x\rangle) \\
             &=  \left \langle \Theta \circ \Phi_2(x) , \Theta \circ \Phi_2(x)\right \rangle. 
        \end{split}
    \end{equation*}
       \end{enumerate}
 \end{remark}

\begin{proposition}\label{PhiPsi}
    Let $\Phi, \Psi \in \mathcal{CP}(E,F).$ Then the following are equivalent:
    \begin{enumerate}
        \item $\Phi \sim \Psi$;
        \item there exists a partial isometry $V \in \mathcal{L_B}(F)$ such that 
        \begin{itemize}
            \item[(i)]  $VV^*={W_\Phi}^*W_\Phi,$ and  $V^*V = {W_\Psi}^*W_\Psi$; 
            \item[(ii)] $\Phi(x) = V\Psi(x)$ for all $x \in E.$
        \end{itemize}
        Here, $W_\Phi$ and $W_\Psi$ are as defined in Theorem \ref{BRmain}.
    \end{enumerate}
        
\end{proposition}

\begin{proof}
    Suppose $\Phi \sim \Psi$. Let $(\pi_\Phi, K_\phi, K_\Phi, V_\phi, W_\Phi)$  and  $(\pi_\Psi, K_\psi, K_\Psi, V_\psi, W_\Psi)$ be the Stinespring constructions associated with $\Phi$ and $\Psi$ respectively.   Then from \cite[Corollary 3.14]{BR}, there exists a unitary operator $U_1: K_\phi \rightarrow K_\psi$ such that $V_\psi = U_1V_\phi.$

    For $x \in E$ and $d \in D,$ we have
    \begin{equation*}
        \begin{split}
            \langle \pi_\Psi(x) V_\psi d,\pi_\Psi(x) V_\psi d  \rangle 
            &= \langle {V_\psi}^*{\pi_\Psi(x)}^*\pi_\Psi(x)V_\psi d,d \rangle \\
            &= \langle {V_\psi}^* \pi_\psi(\langle x, x \rangle) V_\psi d,d \rangle \\
            &= \langle \psi(\langle x, x \rangle) d,d \rangle \\
            &= \langle \phi(\langle x, x \rangle) d,d \rangle \\
            &= \langle \pi_\Phi(x) V_\phi d,\pi_\Phi(x) V_\phi d  \rangle. 
        \end{split}        
    \end{equation*}
     Since $K_\Psi = [\pi_\Psi(X)V_\psi D]$ and $K_\Phi = [\pi_\Phi(X)V_\phi D]$ (see \cite[Remark 3.12]{BR}) , we can define an operator $U_2: K_\Phi \rightarrow K_\Psi$ such that 
    \begin{equation}
        U_2(\pi_\Phi(x)V_\phi d) = \pi_\Psi(x)V_\psi d,
    \end{equation}
    for all $d \in D.$ Clearly $U_2$ is unitary.

    We can easily show that $U_2\pi_\Phi(x) = \pi_\Psi(x)U_1.$ Indeed, using $[\pi_\phi(\mathcal{A})V_\phi D] = K_\phi$ (\cite[Remark 3.12]{BR}),  for  $a \in \mathcal{A}$ and  $d \in D,$ we have       
    \begin{equation*}
    \begin{split}
        U_2\pi_\Phi(x)\big( \pi_\phi(a)V_\phi d \big) &= U_2\big(\pi_\Phi(xa)V_\phi d \big) \\
        &= \pi_\Psi(xa)V_\psi d \\
        &= \pi_\Psi(x)\big(\pi_\psi(a)V_\psi d\big) \\
         &= \pi_\Psi(x)U_1\big(\pi_\phi(a)V_\phi d\big).      
    \end{split}
    \end{equation*}

Define $V= {W_\Phi}^*{U_2}^*W_\Psi$. Then,
$$VV^* = {W_\Phi}^*{U_2}^*W_\Psi{W_\Psi}^*U_2 W_\Phi = {W_\Phi}^*W_\Phi$$ and 
$$V^*V = {W_\Psi}^*U_2W_\Phi{W_\Phi}^*{U_2}^* W_\Psi = {W_\Psi}^*W_\Psi.$$
Note that $V$ is a partial isometry (see \cite[Proposition 3.1.4]{MJbook}).

Hence, for all $x \in E,$ we have
\begin{equation*}
    \begin{split}
        \Phi(x) &= {W_\Phi}^*\pi_\Phi(x)V_\phi (1_\mathcal{B} \otimes \xi) \\
        &= {W_\Phi}^*\pi_\Phi(x){U_1}^*V_\psi (1_\mathcal{B} \otimes \xi) \\
        &= {W_\Phi}^*{U_2}^*\pi_\Psi(x)V_\psi (1_\mathcal{B} \otimes \xi) \\
        &= {W_\Phi}^*{U_2}^*W_\Psi {W_\Psi}^*\pi_\Psi(x)V_\psi (1_\mathcal{B} \otimes \xi) \\
        &= V\Psi(x).
    \end{split}
\end{equation*}

Conversely, suppose there exists an operator $V \in \mathcal{L_B}(F)$ satisfying conditions (i) and (ii). Then, for $x \in E,$ we have
\begin{equation}
    \begin{split}
        \left\langle \Phi(x), \Phi(x) \right\rangle &= \left\langle V\Psi(x), V\Psi(x) \right\rangle \\
       &= \left\langle {W_\Psi}^*W_\Psi\Psi(x), \Psi(x) 
       \right\rangle \\
       &=\left\langle {W_\Psi}^*W_\Psi{W_\Psi}^*\pi_\Psi(x)V_\psi (1_\mathcal{B} \otimes \xi), \Psi(x) \right\rangle \\
       &=\left\langle {W_\Psi}^*\pi_\Psi(x)V_\psi (1_\mathcal{B} \otimes \xi), \Psi(x) \right\rangle \\
        &= \left\langle \Psi(x), \Psi(x) \right\rangle.
    \end{split}
\end{equation}
Hence, $\Phi \sim \Psi$.
 
\end{proof}

\begin{corollary}
 Let $\Phi, \Psi \in \mathcal{CP}(E,F).$ 
 Then the following are equivalent:
 \begin{enumerate}
     \item  $\Phi \sim \Psi$; 
     \item  Stinespring's constructions of $\Phi$ and $\Psi$ are related in the following manner:
     \begin{enumerate}
         \item[(i)] $V_\psi = U_1V_\phi$;
         \item[(ii)] $U_2\pi_\Phi(.) = \pi_\Psi(.)U_1$;
         \item[(iii)] $W_\Phi= {U_2}^*W_\Psi V^*,$ where $V$ is defined as in Proposition \ref{PhiPsi}.
     \end{enumerate}
 \end{enumerate}
\end{corollary}

\begin{proof}
 First assume that $\Phi \sim \Psi$. 
 Let $(\pi_\Phi, K_\phi, K_\Phi, V_\phi, W_\Phi)$ and $(\pi_\Psi, K_\psi$, $K_\Psi, V_\psi, W_\Psi)$ be the Stinespring constructions associated with $\Phi$ and $\Psi$ respectively. Let $U_1, U_2$ be the unitaries as defined in the proof of Proposition \ref{PhiPsi}. Then $V_\psi = U_1V_\phi$ and $U_2\pi_\Phi(x) = \pi_\Psi(x)U_1$, for all $x \in E.$ Morever, with $V$ defined as in Proposition \ref{PhiPsi}, we have $W_\Phi= {U_2}^*W_\Psi V^*.$
    Indeed, post multiplying both sides of $W_\Phi V= {U_2}^* W_\Psi$ by $V^*$, we get $ W_\Phi{W_\Phi}^* W_\Phi= {U_2}^*W_\Psi V^*$ or $W_\Phi= {U_2}^*W_\Psi V^*.$ 

The implication 2. $\Rightarrow$ 1. can be proved by a routine calculation using the given hypothesis.

\end{proof}

We provide the following examples to illustrate the construction of a partial isometry $V$ as described in Proposition \ref{PhiPsi}.

\begin{example} 
   Let $G$ be a Hilbert $\mathcal{A}$-module. It is known that $G^n := \underbrace{G \oplus \cdots \oplus G}_{n\ \text{times}}$ is also a Hilbert $\mathcal{A}$-module. 
   Moreover, $\mathcal{L_A}(G^2, G^5)$ is a Hilbert $\mathcal{L_A}(G^2)$-module. In fact, the pro-$C^*$-algebras $M_2(\mathcal{L_A}(G))$ and $\mathcal{L_A}(G^2)$ are isomorphic. Similarly, we can observe that $M_{5 \times 2}(\mathcal{L_A}(G))$ is identified with $\mathcal{L_A}(G^2, G^5)$ (for further details, see \cite[Corollary 2.2.9]{MJbook}).
   
    Define $\Phi, \Psi: M_2(\mathcal{L_A}(G)) \rightarrow M_{5 \times 2}(\mathcal{L_A}(G))$  by,
    $$ \Phi\left(\begin{pmatrix}
        T_1 & T_2 \\
        T_3 & T_4
    \end{pmatrix}\right) = 
    \begin{pmatrix}
        \frac{1}{2}T_1 & 0 \\
        0 & T_4 \\
        0 & 0 \\
        \frac{1}{2}T_3 & 0 \\
        0 & T_2
    \end{pmatrix},$$  
    and
    $$ \Psi\left(\begin{pmatrix}
        T_1 & T_2 \\
        T_3 & T4
    \end{pmatrix}\right) = 
    \begin{pmatrix}
        \frac{1}{2\sqrt{2}}T_1 &  -\frac{1}{\sqrt{3}}T_4 \\
         \frac{1}{2\sqrt{2}}T_1 &  \frac{1}{\sqrt{3}}T_4 \\
        0 & T_2 \\
        \frac{1}{2}T_3 & 0 \\
        0 & \frac{1}{\sqrt{3}}T_4
    \end{pmatrix}, $$
      for all $\begin{pmatrix}
        T_1 & T_2 \\
        T_3 & T_4
    \end{pmatrix}$ $\in M_2\left(\mathcal{L_A}(G)\right).$
    
Define $\phi :  M_2(\mathcal{L_A}(G)) \rightarrow  M_2(\mathcal{L_A}(G))$ by 
$$\phi\left(\begin{pmatrix}
        T_1 & T_2 \\
        T_3 & T_4
    \end{pmatrix}\right) = \begin{pmatrix}
        \frac{1}{4}T_1 & 0 \\
        0 & T_4
    \end{pmatrix},$$
    for all $\begin{pmatrix}
        T_1 & T_2 \\
        T_3 & T_4
    \end{pmatrix}$ $\in M_2\left(\mathcal{L_A}(G)\right).$
    
Observe that, for any $S, T \in \mathcal{L_A}(G)$, $$\langle \Phi(S), \Phi(T) \rangle = \phi(\langle S, T \rangle) = \langle \Psi(S), \Psi(T) \rangle.$$
Since the underlying map $\phi$ is completely positive, the maps $\Phi, \Psi$ are completely positive. Note that, the map $\Phi$ is degenerate (that is $[\Phi\left(M_2(\mathcal{L_A}(G))\right)(G^2)] \neq G^5$). 

Observe that $M_{5\times2}(\mathcal{L_A}(G))$ is a left $M_5(\mathcal{L_A}(G))$-module. So, we define an operator $V: M_{5 \times 2}(\mathcal{L_A}(G)) \rightarrow  M_{5 \times 2}(\mathcal{L_A}(G))$ by 
$$V\left( \begin{pmatrix}
        T_1 &  T_2 \\
         T_3 &  T_4 \\
        T_5 & T_6 \\
        T_7 & T_8 \\
        T_9 & T_{10}
    \end{pmatrix}\right) := 
     \begin{pmatrix}
        \frac{1}{\sqrt{2}}1_\mathcal{A} & \frac{1}{\sqrt{2}}1_\mathcal{A} & 0 & 0 & 0 \\
         0 & 0 & 0 & 0 & \sqrt{3}1_\mathcal{A}\\
         0 & 0 & 0 &0 & 0 \\
         0 & 0 & 0 & 1_\mathcal{A} & 0 \\
         0 & 0 & 1_\mathcal{A} & 0 & 0 \\
    \end{pmatrix} \begin{pmatrix}
        T_1 &  T_2 \\
         T_3 &  T_4 \\
        T_5 & T_6 \\
        T_7 & T_8 \\
        T_9 & T_{10}
    \end{pmatrix}, $$
for all $\begin{pmatrix}
        T_1 &  T_2 \\
         T_3 &  T_4 \\
        T_5 & T_6 \\
        T_7 & T_8 \\
        T_9 & T_{10}
    \end{pmatrix} \in M_{5\times2}(\mathcal{L_A}(G)).$

Hence, $$\Phi\left(\begin{pmatrix}
        T_1 & T_2 \\
        T_3 & T_4
    \end{pmatrix}\right) = 
   V
    \Psi\left(\begin{pmatrix}
        T_1 & T_2 \\
        T_3 & T_4
    \end{pmatrix}\right),$$
     for all $\begin{pmatrix}
        T_1 & T_2 \\
        T_3 & T_4
    \end{pmatrix}$ $\in M_2\left(\mathcal{L_A}(G)\right).$

\end{example}

Next we see an example in which the maps $\Phi, \Psi $ are non-degenerate.

\begin{example} 
    Let $G$ be a Hilbert $\mathcal{A}$-module. Define $\Phi, \Psi: M_2(\mathcal{L_A}(G)) \rightarrow M_{4 \times 2}(\mathcal{L_A}(G))$  by,
    $$ \Phi\left(\begin{pmatrix}
        T_1 & T_2 \\
        T_3 & T_4
    \end{pmatrix}\right) = 
    \begin{pmatrix}
        \sqrt{2}T_1 & \sqrt{2}T_2 \\
        -T_1 & T_2 \\
        \sqrt{2}T_3 & \sqrt{2}T_4 \\
        -T_3 & T_4 
    \end{pmatrix}$$
    and
    $$ \Psi\left(\begin{pmatrix}
        T_1 & T_2 \\
        T_3 & T_4
    \end{pmatrix}\right) = 
     \begin{pmatrix}
        \sqrt{2}T_1 & \sqrt{2}T_2 \\
        T_1 & -T_2 \\
        \sqrt{2}T_3 & \sqrt{2}T_4 \\
        -T_3 & T_4 
    \end{pmatrix}, $$
     for all $\begin{pmatrix}
        T_1 & T_2 \\
        T_3 & T_4
    \end{pmatrix}$ $\in M_2\left(\mathcal{L_A}(G)\right).$
    
Define $\phi :  M_2(\mathcal{L_A}(G)) \rightarrow  M_2(\mathcal{L_A}(G))$ by 
$$\phi\left(\begin{pmatrix}
        T_1 & T_2 \\
        T_3 & T_4
    \end{pmatrix}\right) = \begin{pmatrix}
        3T_1 & T_2 \\
        T_3 & 3T_4
    \end{pmatrix},$$
    for all $\begin{pmatrix}
        T_1 & T_2 \\
        T_3 & T_4
    \end{pmatrix}$ $\in M_2\left(\mathcal{L_A}(G)\right).$
    
Note that $\phi$ is completely positive, and since $\Phi$ and $\Psi$ are $\phi$-maps, both $\Phi, \Psi$ are completely positive. Observe that, the maps $\Phi, \Psi$ are non-degenerate. 

Hence, $$\Phi\left(\begin{pmatrix}
        T_1 & T_2 \\
        T_3 & T_4
    \end{pmatrix}\right) = 
    \begin{pmatrix}
        1_\mathcal{A} & 0 & 0 & 0  \\
         0 & -1_\mathcal{A} & 0 & 0 \\
         0 & 0 & 1_\mathcal{A} &0  \\
         0 & 0 & 0 & 1_\mathcal{A} 
    \end{pmatrix}
    \Psi\left(\begin{pmatrix}
        T_1 & T_2 \\
        T_3 & T_4
    \end{pmatrix}\right),$$
     for all $\begin{pmatrix}
        T_1 & T_2 \\
        T_3 & T_4
    \end{pmatrix}$ $\in M_2\left(\mathcal{L_A}(G)\right).$

\end{example}

Next we characterize positive elements in $M_n(\mathcal{A}),$ which generalizes  \cite[Lemma 3.13]{VP} for the case of pro-$C^*$-algebras. The proof goes in the similar lines of \cite[Lemma 3.13]{VP}, for the sake of completeness, we provide a proof here.

\begin{lemma}\label{equivdef}
Every positive element in \( M_n(\mathcal{A}) \) can be written as a finite sum of elements of the form \( (a_i^* a_j)_{i,j=1}^n \), where \( a_1, a_2, \dots, a_n \in \mathcal{A} \).
\end{lemma}

\begin{proof}
    If $B \in M_n(\mathcal{A})$ be such that its $k^{\text{th}}$ row is $(a_1, a_2, \dots, a_n)$ and all other entries are zero, then, by definition, $B^*B = ({a_i}^*a_j)$ is positive. Now, let $P$ be a positive element in $ M_n(\mathcal{A}).$ Then $P=Q^*Q$ for some $Q \in M_n(\mathcal{A}).$
    For each $i \in \{1, 2, \dots, n\},$
    let $B_i$ denote the matrix with $i^{\text{th}}$ row of $Q$ as its $i^{\text{th}}$ row and zero elsewhere. Then $Q = B_1 + B_2 + \dots + B_n$. Observe that
    $$P= Q^*Q =  {B_1}^*B_1 + {B_2}^*B_2 + \dots + {B_n}^*B_n.$$
    Indeed ${B_i}^*B_j = 0,$ for all $i,j \in \{1, 2, \dots, n\}$ with $i\neq j$.
\end{proof}

\begin{remark}\label{positivityequi}
    Recall that if
    $\phi: \mathcal{A} \rightarrow \mathcal{B}$ is completely positive, then $\phi^{(n)}: M_n(\mathcal{A}) \rightarrow M_n(\mathcal{B})$ defined by $$\phi^{(n)}([a_{ij}]_{i,j=1}^n) = [\phi(a_{ij})]_{i,j=1}^n,$$ is positive in $M_n(\mathcal{B})$, for each $n \in \mathbb{N}.$ 
    Lemma \ref{equivdef} simplifies the verification of  the positivity of the matrix $[\phi(a_{ij})]_{i,j=1}^n$  to checking that $\left[\phi\left({a_i}^*a_j\right)\right]_{i,j=1}^n$ is positive for all $a_1, a_2, \dots, a_n \in \mathcal{A}.$


If $D$ is a two-sided Hilbert $\mathcal{B}$-module, then for each $n \in \mathbb{N}$, the equivalent condition for verifying the positivity of
    $\phi^{(n)}([a_{ij}]_{i,j=1}^n){I_{D^n}}$ in $M_n(\mathcal{D})$ is to check that $\left[\phi\left({a_i}^*a_j\right)\right]_{i,j=1}^n$ is positive in $M_n(\mathcal{B})$, for all $a_1, a_2, \dots, a_n \in \mathcal{A}.$ That is, for $a_i \in \mathcal{A}$ and $i = 1,\dots, n$, 
    $$\left\langle \begin{pmatrix}
         	d_1 \\ \vdots \\ d_n
         \end{pmatrix} , \left[\tau\left(\phi\left({a_i}^*a_j\right)\right)\right]_{i,j=1}^n  
    \begin{pmatrix}
         d_1 \\ \vdots \\ d_n
     \end{pmatrix}  \right\rangle, \\$$
     is positive  for each $n \in \mathbb{N}.$ Here $\tau: \mathcal{B} \rightarrow \mathcal{L_B}(D)$ is a map such that we can identify $b.d$ with $\tau(b)(d)$ for $b \in \mathcal{B}, d \in D$ (see Definition \ref{leftaction}).
\end{remark}
This equivalent condition for positivity will be frequently applied in the subsequent results. We now introduce a pre-order on the set $\mathcal{CP}(E, F)$.

\begin{definition}
    Let $\Phi, \Psi \in \mathcal{CP}(E, F).$ We define a relation $``\preceq"$ on $\mathcal{CP}(E, F)$ as follows: 
    $$\Psi \preceq	\Phi \text{ if } \phi - \psi \text{ is completely positive.}$$
\end{definition}

The following remark justifies that the above relation is a pre-order on $\mathcal{CP}(E, F).$

\begin{remark} The relation $``\preceq"$ defined above satisfies the following properties:
    \begin{enumerate}
        \item $\Phi \preceq \Phi$ for all $\Phi \in \mathcal{CP}(E,F);$
        \item for $\Phi_1, \Phi_2 \text{ and } \Phi_3 \in \mathcal{CP}(E, F),$ if $\Phi_1 \preceq \Phi_2$ and $\Phi_2 \preceq \Phi_3$ then $\Phi_1 \preceq \Phi_3;$
        \item for $\Phi \text{ and } \Psi \in \mathcal{CP}(E,F),$  we have $\Phi \preceq \Psi$ and $\Psi \preceq \Phi$ if and only if $\Phi \sim \Psi.$
 \end{enumerate}
\end{remark}

Next, we introduce the notion of commutants, inspired by \cite[Definition 4.1]{AL}.

\begin{definition}
Let $G$ be a Hilbert $\mathcal{A}$-module  and $F_1, F_2$ be Hilbert $\mathcal{B}$-modules.    Let $\pi: \mathcal{A} \rightarrow \mathcal{L_B}(F_1)$ be a unital continuous $^*-$morphism and $\Pi: G \rightarrow \mathcal{L_B}(F_1,F_2)$ be a $\pi-$map. We define the commutant of the set $\Pi(G)$ as the set
\[
\begin{aligned}
\Pi(G)' := \{T_1 \oplus T_2 \in \mathcal{L}_{\mathcal{B}}(F_1 \oplus F_2) :\ 
& \Pi(x)T_1 = T_2\Pi(x), \text{ and } \\
& T_1\Pi(x)^* = \Pi(x)^*T_2, \text{ for all } x \in G \}.
\end{aligned}
\]
Here, $(T_1\oplus T_2) (f_1 \oplus f_2) = T_1(f_1) \oplus T_2(f_2)$, for $f_1 \oplus f_2 \in F_1 \oplus F_2.$

Note that $$\pi(\mathcal{A})':= \{T \in \mathcal{L_B}(F_1): \pi(a)T = T\pi(a),\text{ for all } a \in \mathcal{A}\}.$$
\end{definition}

\begin{remark}\label{nondeg} We make the following observations based on the definition above.
    \begin{enumerate}
        \item $\Pi(G)'$ forms a pro-$C^*$-algebra. Indeed for  $T_1\oplus T_2, S_1\oplus S_2 \in \Pi(G)'$ and $\alpha \in \mathbb{C},$  we have 
        $(T_1 + S_1) \oplus (T_2 + S_2), \alpha T_1 \oplus \alpha T_2, T_1^* \oplus T_2^*$ and $T_1 S_1 \oplus T_2 S_2$, all belong to $\Pi(G)'.$
Moreover, $\Pi(G)'$ is closed in $\mathcal{L_{{\mathcal B}}}(F_1 \oplus F_2)$ with respect to the semi-norms on $\mathcal{L_{{\mathcal B}}}(F_1 \oplus F_2)$.
        The proof is similar to \cite[Lemma 4.3]{AL}.
        \item If $[\Pi(G)(F_1)] = F_2,$ (that is, $\Pi$ is non-degenerate) and if  $T_1 \oplus T_2 \in \Pi(G)'$ then $T_2$ is uniquely determined by $T_1.$
        \item Let $G$ be full. If $T_1 \oplus T_2 \in \Pi(G)'$, then $T_1 \in \pi(\mathcal{A})'.$ Indeed, for $x \in G, $ we see that 
        \begin{equation*}
            \begin{split}
                \pi(\langle x,x \rangle)T_1 = \Pi(x)^*\Pi(x)T_1 &=\Pi(x)^*T_2\Pi(x) \\
                &= T_1\Pi(x)^*\Pi(x) \\
                &= T_1\pi(\langle x,x \rangle).
            \end{split}
        \end{equation*}
    \end{enumerate}
\end{remark}

Let 
$$\mathcal{CP}(\mathcal{A},\mathcal{B}):= \{\phi: \mathcal{A} \rightarrow \mathcal{B}: \phi \text{ is continuous, completely positive}\}.$$

We know that whenever $\phi \in \mathcal{CP}(\mathcal{A},\mathcal{B}),$ there is a two-sided Hilbert $\mathcal{B}$-module $D$ and a Stinespring triple $(\pi_\phi, K_\phi, V_\phi)$ attached to it by Theorem \ref{BRmain}.

\begin{lemma}\label{contraction}
     Suppose $\phi, \psi: \mathcal{A} \rightarrow \mathcal{B}$ are two completely positive, continuous maps such that $\psi \leq \phi$. Let $\left(\pi_{\phi}, V_{\phi}, K_{\phi}\right)$ and $\left(\pi_{\psi}, V_{\psi}, K_{\psi}\right)$ be the Stinespring triples associated with $\phi$ and $\psi$ respectively. Then there exists a contraction $J_{\phi}(\psi): K_{\phi} \rightarrow K_{\psi}$ such that:
    \begin{enumerate}
        \item $J_{\phi}(\psi) V_{\phi} = V_{\psi}$;
        \item $J_{\phi}(\psi) \pi_{\phi}(a) = \pi_{\psi}(a) J_{\phi}(\psi)$, for all $a \in \mathcal{A}$.
    \end{enumerate}
\end{lemma}

\begin{proof}
   Define a linear map $J_{\phi}(\psi): K_\phi \rightarrow K_\psi$ by 
    $$J_\phi(\psi)\left(\pi_\phi(a)V_\phi d\right) = \pi_\psi(a)V_\psi d,$$
    for all $a \in \mathcal{A}$ and $d \in D.$

    Given that $\phi - \psi$ is completely positive, we observe that, for $a_1,\dots a_n \in \mathcal{A}$ and $d_1, \dots, d_n \in D,$ we have 
\begin{equation*}
	\begin{split}
& \left\langle J_{\phi}(\psi)\left(\sum_{i=1}^n\pi_{\phi}(a_i)V_\phi d_i\right), J_{\phi}(\psi)  \left(\sum_{i=1}^n\pi_{\phi}(a_i)V_\phi d_i\right) \right\rangle \\ 
&=  \sum_{i,j=1}^n \left \langle \pi_\psi(a_i)V_\psi d_i, \pi_\psi(a_j)V_\psi d_j \right\rangle   \\
&= \sum_{i,j=1}^n \left \langle V_\psi d_i, \pi_\psi({a_i}^*a_j)V_\psi d_j \right\rangle   \\
&= \sum_{i,j=1}^n \left \langle  d_i, {V_\psi}^*\pi_\psi({a_i}^*a_j)V_\psi d_j \right\rangle   \\
&= \sum_{i,j=1}^n \left \langle  d_i, \psi({a_i}^*a_j) d_j \right\rangle   \\
&= \left\langle \begin{pmatrix}
         	d_1 \\ \vdots \\ d_n
         \end{pmatrix} , \left[\tau\left(\psi\left({a_i}^*a_j\right)\right)\right]_{i,j=1}^n  
    \begin{pmatrix}
         d_1 \\ \vdots \\ d_n
     \end{pmatrix}  \right\rangle \\
&\leq \left\langle \begin{pmatrix}
         	d_1 \\ \vdots \\ d_n
         \end{pmatrix} , \left[\tau\left(\phi\left({a_i}^*a_j\right)\right)\right]_{i,j=1}^n  
    \begin{pmatrix}
         d_1 \\ \vdots \\ d_n
     \end{pmatrix}  \right\rangle \\
&= \sum_{i,j=1}^n \left \langle  d_i, \phi({a_i}^*a_j) d_j \right\rangle   \\
&=  \left\langle \sum_{i=1}^n\pi_{\phi}(a_i)V_\phi d_i, \sum_{i=1}^n\pi_{\phi}(a_i)V_\phi d_i \right \rangle .
	\end{split}
\end{equation*}
Thus, $\|J_{\phi}(\psi)\| \leq 1.$ 
Since $[\pi_\phi(\mathcal{A})V_\phi D] = K_\phi,$ we can uniquely extend this operator to an operator from $K_\phi$ to $K_\psi$.

    The proof of 1. and 2. is analogous to that of \cite[Theorem 3.13]{BR}.

\end{proof}

Define an interval $[0,\phi]$ by
$$[0,\phi]:= \{\psi \in \mathcal{CP}(\mathcal{A},\mathcal{B}): \psi \leq \phi\}.$$
For each $T \in \pi_\phi(\mathcal{A})'$, define a linear map $\phi_T: \mathcal{A} \rightarrow \mathcal{B}$  given by
    $$\phi_T(a)I_D= {V_\phi}^*T\pi_\phi(a)V_\phi,$$
     for all $a \in \mathcal{A}.$
     
We now establish that, for $\phi, \psi \in \mathcal{CP}(\mathcal{A},\mathcal{B})$, the relation $\psi \leq \phi$ holds if and only if there exists a positive contraction $T \in \pi_\phi(\mathcal{A})'$ satisfying
\[
\psi(\cdot)I_D = V_\phi^* T \pi_\phi(\cdot) V_\phi.
\]
This result can be regarded as a Radon–Nikod\'ym-type theorem for completely positive maps between pro-$C^*$-algebras.

\begin{theorem}\label{phiT}
   The map $T \mapsto \phi_T$ is an affine order isomorphism from the set $\{T \in \pi_\phi(\mathcal{A})': 0 \leq T \leq I\}$ onto $[0,\phi].$
\end{theorem}

\begin{proof}
     Let $T  \in \pi_\phi(\mathcal{A})'$ such that $0\leq T\leq I$. Clearly the map $\phi_T$ is linear for each $T.$ Morever, for $S  \in \pi_\phi(\mathcal{A})'$ such that $0\leq S\leq I$, and for $\lambda \in [0,1]$, we have
     \begin{equation*}
         \begin{split}
             \phi_{\lambda T + (1-\lambda)S}(\cdot)I_D &= {V_\phi}^*\left(\lambda T + (1-\lambda)S\right)\pi_\phi(\cdot)V_\phi \\
             &= \lambda{V_\phi}^* T\pi_\phi(\cdot)V_\phi + (1-\lambda){V_\phi}^*S\pi_\phi(\cdot)V_\phi \\
             &= \left(\lambda \phi_{T}(\cdot) + (1-\lambda)\phi_{S}(\cdot)\right)I_D.
         \end{split}
     \end{equation*}
     This shows that the map $T \mapsto \phi_T$ is affine.
     
    Next, we observe that $\phi_T$ is completely positive. Indeed, since $\pi_\phi$ is completely positive, for $a_1,\dots a_n \in \mathcal{A}$ and $d_1, \dots, d_n \in D,$  we have
 \begin{equation*}
     \begin{split}
         \left\langle \begin{pmatrix}
         	d_1 \\ \vdots \\ d_n
         \end{pmatrix} , \left[\tau\left(\phi_T\left({a_i}^*a_j\right)\right)\right]_{i,j=1}^n  
    \begin{pmatrix}
         d_1 \\ \vdots \\ d_n
     \end{pmatrix}  \right\rangle &= \sum_{i,j=1}^n \left \langle d_i, \tau\left(\phi_T\left({a_i}^*a_j\right)\right)d_j \right \rangle \\
     &= \sum_{i,j=1}^n \left \langle d_i, {V_\phi}^*T\pi_\phi\left({a_i}^*a_j\right)V_\phi d_j \right \rangle \\
     &= \sum_{i,j=1}^n \left \langle {T}^\frac{1}{2}V_\phi d_i, T^\frac{1}{2}\pi_\phi\left({a_i}^*a_j\right)V_\phi d_j \right \rangle \\
     &= \sum_{i,j=1}^n \left \langle {T}^\frac{1}{2}V_\phi d_i, \pi_\phi\left({a_i}^*a_j\right)T^\frac{1}{2}V_\phi d_j \right \rangle, 
     \end{split}
 \end{equation*}
 is positive  for each $n \in \mathbb{N}.$ Thus, by Remark \ref{positivityequi}, ${\phi_T}^n([a_{ij}]_{i,j=1}^n)$ is positive in $M_n(\mathcal{B})$ for each $n \in \mathbb{N}$.

Replacing $T$ by $I-T$ in the above calculations, we in fact get $\phi_T \in [0,\phi].$ 
In the same manner, for $T_1,T_2  \in \pi_\phi(\mathcal{A})'$ such that $0\leq T_1 \leq T_2 \leq I$, we get $\phi_{T_1} \leq \phi_{T_2} \leq \phi.$

Next we show that the map $T \mapsto \phi_T$ is injective. If $\phi_T = 0$, then by order relation $\phi_{T^2}=0$. For $a \in \mathcal{A}$ and $d\in D$, we have
\begin{equation*}
    \begin{split}
        \left \langle T\pi_\phi(a)V_\phi d, T\pi_\phi(a)V_\phi d \right \rangle &= \left \langle  d, {V_\phi}^*\pi_\phi({a}^*)T^2\pi_\phi(a)V_\phi d \right \rangle \\
        &=  \left \langle  d, {V_\phi}^*T^2\pi_\phi({a}^*a)V_\phi d \right \rangle \\
        &= \left \langle  d, \phi_{T^2}({a}^*a) d \right \rangle \\
        &= 0.
    \end{split}
\end{equation*}
Since $[\pi_\phi(\mathcal{A})V_\phi(D)] = K_\phi$, we have $T=0$. Hence the map $T \mapsto \phi_T$ is injective.

Let $T \in \pi_\phi(\mathcal{A})'$. If $\phi_T \geq 0,$ we show that $T \geq 0.$ If $k_\phi \in K_\phi,$ then $k_\phi = \underset{i=1}{\overset{n}{\sum}}\pi_\phi(a_i)V_\phi d_i$, where $a_i \in \mathcal{A}$ and $ d_i \in D$ for all $i = 1,\dots, n$. Observe that,
\begin{equation*}
    \begin{split}
        \langle  k_\phi, T k_\phi \rangle &= \left \langle  \sum_{i=1}^n\pi_\phi(a_i)V_\phi d_i ,T\left( \sum_{i=1}^n\pi_\phi(a_i)V_\phi d_i\right)  \right \rangle \\
        &= \sum_{i,j = 1}^n \left \langle d_i, {V_\phi}^*T\pi_\phi({a_i}^*a_j)V_\phi d_j \right \rangle \\
        &= \sum_{i,j = 1}^n \left \langle d_i, \phi_T({a_i}^*a_j) d_j \right \rangle \\
        &\geq 0.
    \end{split}
\end{equation*}
In the similar way, for  $T_1,T_2  \in \pi_\phi(\mathcal{A})'$ such that $0\leq \phi_{T_1} \leq \phi_{T_2} \leq \phi,$ we get $0 \leq T_1 \leq T_2 \leq I.$ 

Let $\psi \in [0,\phi]$, and let $\left(\pi_\psi, V_\psi, K_\psi \right)$ be the Stinespring triple associated with $\psi$. Then by Lemma \ref{contraction} there exists a contraction $J_{\phi}(\psi): K_{\phi} \rightarrow K_{\psi}$ such that
    $J_{\phi}(\psi) V_{\phi} = V_{\psi}$ and
       $J_{\phi}(\psi) \pi_{\phi}(a) = \pi_{\psi}(a) J_{\phi}(\psi)$, for all $a \in \mathcal{A}$.
   Define $T:=J_{\phi}(\psi)^*J_{\phi}(\psi)$. Clearly $0 \leq T \leq I$ and $T \in \pi_\phi(\mathcal{A})'.$ Indeed, for $a \in \mathcal{A}$, we have  $$J_{\phi}(\psi)^*J_{\phi}(\psi)\pi_\phi(a^*) = J_{\phi}(\psi)^*\pi_{\psi}(a^*) J_{\phi}(\psi) = \pi_{\phi}(a^*)J_{\phi}(\psi)^*J_{\phi}(\psi).$$ Finally, for $a \in \mathcal{A}, d_1,d_2 \in D$, observe that
   \begin{equation*}
       \begin{split}
           \langle d_1, \phi_T(a)d_2 \rangle &= \langle d_1, {V_\phi}^*J_{\phi}(\psi)^*J_{\phi}(\psi) \pi_\phi(a) V_\phi d_2 \rangle \\
           &= \langle J_{\phi}(\psi){V_\phi}d_1, J_{\phi}(\psi) \pi_\phi(a) V_\phi d_2 \rangle \\
           &= \langle J_{\phi}(\psi){V_\phi}d_1, \pi_{\psi}(a) J_{\phi}(\psi)  V_\phi d_2  \rangle \\
           &= \langle V_\psi d_1, \pi_{\psi}(a) V_\psi d_2   \rangle \\
           &= \langle  d_1, \psi(a)d_2 \rangle.
       \end{split}
   \end{equation*}
This shows that the map is surjective, thereby establishing the correspondence.
\end{proof}

Next, for each element in the commutant, we associate a completely positive map, as shown below.

\begin{lemma}\label{rep}
    Let $\Phi \in \mathcal{CP}(E, F)$ and $(\pi_\Phi, K_\phi, K_\Phi, V_\phi, W_\Phi)$ be the Stinespring construction associated with $\Phi$. Let $T \oplus S \in {\pi_\Phi(E)}' $ be a positive element. Then the map $\Phi_{T \oplus S}: E \rightarrow \mathcal{L_B}(K_\phi, K_\Phi)$ defined by 
    \begin{equation*}
    \Phi_{T \oplus S}(x) = {W_\Phi}^*\sqrt{S}\pi_\Phi(x)\sqrt{T}V_\phi (1_\mathcal{B} \otimes \xi),
    \end{equation*}
    is completely positive.
\end{lemma}

\begin{proof}

Let $T \oplus S \in {\pi_\Phi(E)}'$. For $x, y \in E$, we have 

\begin{equation}\label{-1}
    \begin{split}
         \left\langle \Phi_{T \oplus S}(x), &\Phi_{T \oplus S}(y) \right\rangle \\
        &=  \left\langle {W_\Phi}^*\sqrt{S}\pi_\Phi(x)\sqrt{T}V_\phi (1_\mathcal{B} \otimes \xi), {W_\Phi}^*\sqrt{S}\pi_\Phi(y)\sqrt{T}V_\phi (1_\mathcal{B} \otimes \xi) \right\rangle \\
        &= \left\langle \sqrt{S}\pi_\Phi(x)\sqrt{T}V_\phi (1_\mathcal{B} \otimes \xi), \sqrt{S}\pi_\Phi(y)\sqrt{T}V_\phi (1_\mathcal{B} \otimes \xi) \right\rangle \\
        &= \left\langle \sqrt{T}{\pi_\Phi(y)}^*S\pi_\Phi(x)\sqrt{T}V_\phi (1_\mathcal{B} \otimes \xi), V_\phi (1_\mathcal{B} \otimes \xi) \right\rangle \\
        &= \left\langle \sqrt{T}{\pi_\Phi(y)}^*S^{\frac{3}{2}}\pi_\Phi(x)V_\phi (1_\mathcal{B} \otimes \xi), V_\phi (1_\mathcal{B} \otimes \xi)\right\rangle \\
        &=\left\langle T^2{\pi_\Phi(y)}^*\pi_\Phi(x)V_\phi (1_\mathcal{B} \otimes \xi), V_\phi (1_\mathcal{B} \otimes \xi) \right\rangle \\
        &= \left\langle T^2\pi_\phi(\langle y,x \rangle)V_\phi (1_\mathcal{B} \otimes \xi), V_\phi (1_\mathcal{B} \otimes \xi)\right\rangle \\
        &= \left\langle \phi_{T^2}(\langle y,x \rangle)I_D(1_\mathcal{B} \otimes \xi), I_D(1_\mathcal{B} \otimes \xi) \right\rangle \\
        &= \phi_{T^2}(\langle y,x \rangle) \left\langle I_D(1_\mathcal{B} \otimes \xi), I_D(1_\mathcal{B} \otimes \xi) \right\rangle  \\
&= \phi_{T^2}(\langle y,x \rangle) \left\langle  \xi, {1_\mathcal{B}}^*1_\mathcal{B}\xi \right \rangle \\
&= \phi_{T^2}(\langle y,x \rangle) \left\langle  \xi, \xi \right\rangle \\
        &= \phi_{T^2}(\langle x,y \rangle).
    \end{split}
\end{equation}
Indeed, by \cite[Lemma 1]{KS}, $\xi := V_\phi(1_\mathcal{B})$. Observe that,  $\langle V_\phi(1_\mathcal{B}) , V_\phi(1_\mathcal{B}) \rangle = \langle 1_\mathcal{B} \otimes 1_\mathcal{B}, 1_\mathcal{B} \otimes 1_\mathcal{B} \rangle$ by \cite[Theorem 4.6]{MJ0}.

Therefore, $\Phi_{T \oplus S}$ is a $\phi_{T^2}$-map. Since $\phi_{T^2}$ is completely positive, as shown in the proof of Theorem \ref{phiT}, the proof is complete.
\end{proof}

Here, we say that $\phi_{T^2}$ is the completely positive map associated with $\Phi_{T \oplus S}$.

\begin{theorem}\label{main}
    Let $\Psi, \Phi \in \mathcal{CP}(E, F).$ If $\Psi \preceq \Phi$ then there exists a unique positive element $\Delta_\Phi(\Psi) \in \pi_\Phi(E)'$ such that $\Psi \sim \Phi_{\sqrt{\Delta_\Phi(\Psi)}}.$
\end{theorem}

\begin{proof}

 Let $\Psi, \Phi \in \mathcal{CP}(E, F)$ such that $\Psi \preceq \Phi$. Since $\psi \leq \phi,$  by Lemma \ref{contraction} and Theorem \ref{phiT}, there exists a contraction $J_{\phi}(\psi): K_{\phi} \rightarrow K_{\psi}$ such that

\begin{equation}\label{0}
    \phi_{{J_{\phi}(\psi)}^*J_{\phi}(\psi)}(a) = \psi(a),
\end{equation}
for all $a \in \mathcal{A}.$   

Next, define $I_\Phi(\Psi): K_\Phi \rightarrow K_\Psi $ by
\begin{equation*}
I_\Phi(\Psi)\left(\sum_{i=1}^n\pi_{\Phi}(x_i)V_\phi d_i\right) = \sum_{i=1}^{n} \pi_{\Psi}(x_i)V_\psi d_i,
\end{equation*}
for all $x_1,\dots x_n \in E$ and $d_1, \dots, d_n \in D, n \geq 1.$ 
Observe that, for $x_1,\dots x_n \in E$ and $d_1, \dots, d_n \in D,$ we have

\begin{equation*}
	\begin{split}
	& \left\langle I_\Phi(\Psi)\left(\sum_{i=1}^n\pi_{\Phi}(x_i)V_\phi d_i\right), I_\Phi(\Psi)  \left(\sum_{i=1}^n\pi_{\Phi}(x_i)V_\phi d_i\right) \right\rangle \\
    &=  \sum_{i,j=1}^n \left \langle \pi_{\Psi}(x_i)V_\psi d_i, \pi_{\Psi}(x_j)V_\psi d_j \right\rangle   \\
		&= \sum_{i,j=1}^n \left \langle V_\psi d_i, \pi_\Psi(x_i)^*\pi_{\Psi}(x_j)V_\psi d_j \right\rangle   \\
		&= \sum_{i,j=1}^n \left \langle  d_i, {V_\psi}^*\langle \pi_\Psi(x_i), \pi_{\Psi}(x_j) \rangle V_\psi d_j \right\rangle   \\
		&= \sum_{i,j=1}^n \left \langle  d_i, {V_\psi}^* \pi_\psi(\langle x_i, x_j \rangle) V_\psi d_j \right\rangle   \\
&= \left\langle \begin{pmatrix}
         	d_1 \\ \vdots \\ d_n
         \end{pmatrix} , \left[\tau\left(\psi\left(\langle x_i, x_j \rangle\right)\right)\right]_{i,j=1}^n  
    \begin{pmatrix}
         d_1 \\ \vdots \\ d_n
     \end{pmatrix}  \right\rangle \\
&\leq \left\langle \begin{pmatrix}
         	d_1 \\ \vdots \\ d_n
         \end{pmatrix} , \left[\tau\left(\phi\left(\langle x_i, x_j \rangle\right)\right)\right]_{i,j=1}^n  
    \begin{pmatrix}
         d_1 \\ \vdots \\ d_n
     \end{pmatrix}  \right\rangle \\
		&= \sum_{i,j=1}^n \left \langle  d_i, {V_\phi}^* \pi_\phi(\langle x_i, x_j \rangle) V_\phi d_j \right\rangle   \\
		&=  \left\langle \sum_{i=1}^n\pi_\Phi(x_i)V_\phi d_i, \sum_{i=1}^n\pi_\Phi(x_i)V_\phi d_i \right\rangle .	
	\end{split}
\end{equation*}
Thus, $\|I_\Phi(\Psi)\| \leq 1.$ Since $[\pi_\Phi(E)V_\phi D] = K_\Phi,$ we can uniquely extend this operator to an operator from $K_\Phi$ to $K_\Psi$. 

For $x \in E, a \in \mathcal{A}$ and $d \in D,$ 
\begin{equation*}
    \begin{split}
        I_\Phi(\Psi)\pi_\Phi(x)(\pi_\phi(a)V_\phi d) &= 
        I_\Phi(\Psi)\pi_\Phi(xa)V_\phi d) \\
        &= \pi_\Psi(xa)V_\psi d \\
        &= \pi_\Psi(x)\pi_\psi(a)V_\psi d \\
        &= \pi_\Psi(x)J_{\phi}(\psi)\left(\pi_\phi(a)V_\phi d\right).
    \end{split}
\end{equation*}
Since $[\pi_\phi(a)V_\phi d] = K_\phi,$ we have 
\begin{equation}\label{1}
I_\Phi(\Psi)\pi_\Phi(x) = \pi_\Psi(x)J_{\phi}(\psi), \text{ for all } x \in E. 
\end{equation}
Similarly, we have 
\begin{equation}\label{2}
{\pi_\Psi(x)}^*I_\Phi(\Psi) = J_{\phi}(\psi){\pi_\Psi(x)}^*, \text{ for all } x \in E. 
\end{equation}
Indeed, since $[\pi_\Phi(x)V_\phi d] = K_\Phi,$ for $x, y \in E,$ and $d \in D,$ observe
\begin{equation*}
    \begin{split}
        {\pi_\Psi(x)}^*I_\Phi(\Psi)(\pi_\Phi(y)V_\phi d) &= 
        {\pi_\Psi(x)}^*(\pi_\Psi(y)V_\psi d) \\
        &= \pi_\psi(\langle x,y\rangle)V_\psi d \\
        &= J_{\phi}(\psi)\left(\pi_\phi(\langle x,y\rangle)V_\phi d\right) \\
        &=J_{\phi}(\psi){\pi_\Phi(x)}^*(\pi_\Phi(y)V_\psi d).
    \end{split}
\end{equation*}

Define $\Delta_\Phi(\Psi) := \Delta_{1 \Phi}(\Psi) \oplus \Delta_{2 \Phi}(\Psi)$, where 
$\Delta_{1 \Phi}(\Psi):= {J_{\phi}(\psi)}^*J_{\phi}(\psi)$ and
$\Delta_{2 \Phi}(\Psi):= {I_\Phi(\Psi)}^*I_\Phi(\Psi).$

Using equations \ref{1} and \ref{2}, for $x \in E$,  we have
\begin{equation*}
    \begin{split}
        \Delta_{2 \Phi}(\Psi)\pi_\Phi(x) = {I_\Phi(\Psi)}^*I_\Phi(\Psi)\pi_\Phi(x) &= {I_\Phi(\Psi)}^*\pi_\Psi(x)J_{\phi}(\psi) \\
        &= {\pi_\Phi(x)}{J_{\phi}(\psi)}^*J_{\phi}(\psi) \\
        &= \pi_\Phi(x)\Delta_{1 \Phi}(\Psi).
    \end{split}
\end{equation*}
Similarly,
\begin{equation*}
    \begin{split}
        {\pi_\Phi(x)}^*\Delta_{2 \Phi}(\Psi) = {\pi_\Phi(x)}^*{I_\Phi(\Psi)}^*I_\Phi(\Psi) &= {J_{\phi}(\psi)}^*{\pi_\Psi(x)}^*I_\Phi(\Psi) \\
        &= {J_{\phi}(\psi)}^*J_{\phi}(\psi){\pi_\Phi(x)}^* \\
        &= \Delta_{1 \Phi}(\Psi){\pi_\Phi(x)}^*,
    \end{split}
\end{equation*}
    for all $x \in E$.

This says that $\Delta_\Phi(\Psi) \in \pi_\Phi(E)'$ and $\|\Delta_\Phi(\Psi)\| \leq 1$.

As seen in Lemma \ref{rep}, we know that the map $\Phi_{\Delta_\Phi(\Psi)},$ given by $$\Phi_{\Delta_\Phi(\Psi)}(x) = {W_\Phi}^*\sqrt{\Delta_{2\Phi}(\Psi)}\pi_\Phi(x)\sqrt{\Delta_{1\Phi}(\Psi)}V_\phi(1_\mathcal{B} \otimes \xi),$$ is completely positive.

Moreover, by equations \ref{-1} and \ref{0}, for $x \in E,$ we have
\begin{equation*}
    \begin{split}
        \left \langle\Phi_{\sqrt{\Delta_\Phi(\Psi)}}(x), \Phi_{\sqrt{\Delta_\Phi(\Psi)}}(x) \right \rangle &= \phi_{\Delta_{1\Phi}}(\langle x,x\rangle) \\
        &= \psi(\langle x,x\rangle) \\
        &= \langle\Psi(x),\Psi(x)\rangle.
    \end{split}
\end{equation*}
Thus, $\Psi \sim \Phi_{\sqrt{\Delta_\Phi(\Psi)}}.$

Next, we show uniqueness of the map $\Delta_\Phi(\Psi).$ Suppose there is another positive linear operator $T \oplus S \in \pi_\Phi(E)'$ such that $\Psi \sim \Phi_{\sqrt{T \oplus S}}.$ 
 Then $ \Phi_{\sqrt{\Delta_\Phi(\Psi)}} \sim \Phi_{\sqrt{T \oplus S}}.$ Hence the associated maps are equal, that is, $\phi_{\Delta_{1\Phi}(\Psi)}(x) = \phi_T.$
By Theorem \ref{phiT}, we know that the map $T \mapsto \phi_T$ is injective. Thus, we get $T= \Delta_{1\Phi}(\Psi).$ Since $S$ is completely determined by $T$, by \cite[Remark 3.12]{BR} and Remark \ref{nondeg}, we obtain $S = \Delta_{2\Phi}(\Psi).$

\end{proof}

Note that the positive linear map $\Delta_\Phi(\Psi) := \Delta_{1 \Phi}(\Psi) \oplus \Delta_{2 \Phi}(\Psi) \in \pi_\Phi(E)'$ will be called as the Radon-Nikod\'ym derivative of $\Psi$ with respect to $\Phi.$

\begin{remark}
    \begin{enumerate}
        \item If $\Delta_\Phi(\Psi) := \Delta_{1 \Phi}(\Psi) \oplus \Delta_{2 \Phi}(\Psi) \in \pi_\Phi(E)'$ is the Radon-Nikod\'ym derivative of $\Psi$ with respect to $\Phi$, then $\Delta_{1 \Phi}(\Psi) \in \pi_\phi(\mathcal{A})'$ is called the Radon-Nikod\'ym derivative of $\psi$ with respect to $\phi$ (as observed in Theorem \ref{phiT}). 
        \item If $\Psi_1 \preceq \Phi$, $\Psi_2 \preceq \Phi$ and $\Psi_1 \sim \Psi_2$ then $\Delta_\Phi(\Psi_1) = \Delta_\Phi(\Psi_2)$. Indeed, $\Psi_1 \sim \Psi_2$ implies $\psi_1 = \psi_2 $ which inherently implies $J_\phi(\psi_1) = J_\phi(\psi_2)$. Since $\Delta_{1 \Phi}(\Psi)$ uniquely determines $\Delta_{2 \Phi}(\Psi),$ we have the required conclusion.
    \end{enumerate}
\end{remark}

\begin{theorem}\label{unieq}
   Let $\Phi, \Psi \in \mathcal{CP}(E, F).$ Let $(\pi_\Phi, K_\phi, K_\Phi, V_\phi, W_\Phi)$ be the Stinepring's construction associated with $\Phi.$ Let  $\Delta_{1 \Phi}(\Psi)$ and $\Delta_{2 \Phi}(\Psi)$ be defined as in Theorem \ref{main}. Suppose $\text{ker}(\Delta_{1 \Phi}(\Psi))$ and $\text{ker}(\Delta_{2 \Phi}(\Psi))$ are complemented. If $\Psi \preceq \Phi$, then there exists a unitarily equivalent Stinespring's construction associated to $\Psi.$
\end{theorem}

\begin{proof}
    We know that $\Delta_\Phi(\Psi) = \Delta_{1 \Phi}(\Psi) \oplus \Delta_{2 \Phi}(\Psi) \in \pi_\Phi(E)'.$ For $x \in E,$ observe that, 
    for $k_\phi \in \text{ker}(\Delta_{1 \Phi}(\Psi)),$
    $$\Delta_{2 \Phi}(\Psi)\pi_\Phi(x)(k_\phi) = \pi_\Phi(x)\Delta_{1 \Phi}(\Psi)(k_\phi) = 0,$$
    and for $k_\Phi \in \text{ker}(\Delta_{2 \Phi}(\Psi))$, we have 
    $$\Delta_{1 \Phi}(\Psi){\pi_\Phi(x)}^*(k_\Phi) = {\pi_\Phi(x)}^*\Delta_{2 \Phi}(\Psi)(k_\Phi) = 0.$$
    Thus, the pair $\big(\text{ker}(\Delta_{1 \Phi}(\Psi)), \text{ker}(\Delta_{2 \Phi}(\Psi))\big)$ is invariant under $\pi_\Phi.$

    Note that, for $x \in E,$ 
        $$\pi_\Phi(x)P_{\text{ker}(\Delta_{1 \Phi}(\Psi))} = P_{\text{ker}(\Delta_{2 \Phi}(\Psi))}\pi_\Phi(x),$$
        and 
       $${\pi_\Phi(x)}^*P_{\text{ker}(\Delta_{2 \Phi}(\Psi))} = P_{\text{ker}(\Delta_{1 \Phi}(\Psi))}{\pi_\Phi(x)}^*.$$
   Indeed, since $\text{ker}(\Delta_{1 \Phi}(\Psi))$ and $\text{ker}(\Delta_{2 \Phi}(\Psi))$ are complemented, $$K_\phi = \text{ker}(\Delta_{1 \Phi}(\Psi)) \oplus {\text{ker}(\Delta_{1 \Phi}(\Psi))}^\perp$$ and $$K_\Phi = \text{ker}(\Delta_{2 \Phi}(\Psi)) \oplus {\text{ker}(\Delta_{2 \Phi}(\Psi)})^\perp.$$ Let $k_\phi = k_{1\phi} \oplus k_{2\phi} \in K_\phi$ and $k_\Phi = k_{1\Phi} \oplus k_{2\Phi} \in K_\Phi$ be such that $\pi_\Phi(x)(k_\phi) = k_\Phi.$ Since $\pi_\Phi(E)(\text{ker}(\Delta_{1 \Phi}(\Psi))) \subseteq \text{ker}(\Delta_{2 \Phi}(\Psi)),$ we have
    $$\pi_\Phi(x)P_{\text{ker}(\Delta_{1 \Phi}(\Psi))}(k_\phi) = k_{1 \Phi} = P_{\text{ker}(\Delta_{2 \Phi}(\Psi))}(k_\Phi).$$
    
\noindent    Similarly, since ${\pi_\Phi(X)}^*(\text{ker}(\Delta_{2 \Phi}(\Psi))) \subseteq \text{ker}(\Delta_{1 \Phi}(\Psi)),$ for $j_\phi = j_{1\phi} \oplus j_{2\phi} \in K_\phi \text{ and } j_\Phi = j_{1\Phi} \oplus j_{2\Phi} \in K_\Phi$, if ${\pi_\Phi(x)}^*(j_\Phi) = j_\phi$, we have
    $${\pi_\Phi(x)}^*P_{\text{ker}(\Delta_{2 \Phi}(\Psi))}(j_\Phi) = j_{1 \phi} = P_{\text{ker}(\Delta_{1 \Phi}(\Psi))}(j_\phi).$$
   This shows that $P_{\text{ker}(\Delta_{1 \Phi}(\Psi))} \oplus P_{\text{ker}(\Delta_{2 \Phi}(\Psi))} \in \pi_\Phi(E)'.$ Similarly, we can observe that $P_{K_\phi \ominus \text{ker}(\Delta_{1 \Phi}(\Psi))} \oplus P_{K_\Phi \ominus \text{ker}(\Delta_{2 \Phi}(\Psi))} \in \pi_\Phi(E)'.$

   Let $P_1= P_{K_\phi \ominus \text{ker}(\Delta_{1 \Phi}(\Psi))} $ and $P_2 = P_{K_\Phi \ominus \text{ker}(\Delta_{2 \Phi}(\Psi))}.$ Then the Stinespring's construction associated to $\Psi$ is unitarily equivalent to $$\left(P_2\pi_\Phi(\cdot)P_1, K_\phi \ominus \text{ker}(\Delta_{1 \Phi}(\Psi)), K_\Phi \ominus \text{ker}(\Delta_{2 \Phi}(\Psi)), P_1\sqrt{\Delta_{1\Phi}(\Psi)}V_\Phi, P_2 W_\Phi \right).$$

   Indeed, for each $x \in E$, $$P_2\pi_\Phi(x)P_1 \in \mathcal{L_B}\left(K_\phi \ominus \text{ker}(\Delta_{1 \Phi}(\Psi)),K_\Phi \ominus \text{ker}(\Delta_{2 \Phi}(\Psi))\right).$$ In fact, 
   \begin{equation*}
       \begin{split}
           \langle P_2\pi_\Phi(x)P_1 , P_2\pi_\Phi(y)P_1 \rangle 
           &= P_1 \pi_\Phi(x)^*P_2\pi_\Phi(y)P_1 \\
           &=  P_1 P_1\pi_\Phi(x)^*\pi_\Phi(y)P_1 \\
           &= P_1\langle\pi_\phi(x),\pi_\phi(y)\rangle P_1,
       \end{split}
   \end{equation*}
for all $x,y \in E.$  Hence $P_2\pi_\Phi(.)P_1$ is a $P_2\pi_\phi(.)P_1-$map.

Note that
$$(P_2 W_\Phi)(P_2 W_\Phi)^* = P_2 W_\Phi {W_\Phi}^* P_2 = P_2,$$
hence $P_2 W_\Phi \in \mathcal{L_B}(F,K_\Phi \ominus \text{ker}(\Delta_{2 \Phi}(\Psi)))$ is a co-isometry.

Observe that 
\begin{equation*}
    \begin{split}
        \left[ P_2\pi_\Phi(\cdot)P_1 \left( P_1\sqrt{\Delta_{1\Phi}(\Psi)}V_\phi \right)D \right] 
        &=  \left[ P_2\pi_\Phi(\cdot)\sqrt{\Delta_{1\Phi}(\Psi)}V_\phi D \right] \\
        &= \left[ P_2\sqrt{\Delta_{2\Phi}(\Psi)}\pi_\Phi(\cdot)V_\phi D \right] \\
        &= \left[ P_2\sqrt{\Delta_{2\Phi}(\Psi)}K_\Phi\right] \\
        &= K_\Phi \ominus \text{ker}(\Delta_{2 \Phi}(\Psi)).
    \end{split}
\end{equation*}
This shows minimality of the construction.


Note that $  \Psi \sim   \Phi_{\sqrt{\Delta_{\Phi}(\Psi)}}$, by Theorem \ref{main}. Finally, we observe that, for all $x \in E,$
\begin{equation*}
    \begin{split}
     \Phi_{\sqrt{\Delta_{\Phi}(\Psi)}}(x) &= {W_\Phi}^*{\Delta_{2\Phi}(\Psi)}^\frac{1}{4}\pi_\Phi(x){\Delta_{1\Phi}(\Psi)}^\frac{1}{4}V_\phi(1_\mathcal{B} \otimes \xi) \\
      &={W_\Phi}^*\pi_\Phi(x)\sqrt{\Delta_{1\Phi}(\Psi)}V_\phi(1_\mathcal{B} \otimes \xi) \\
      &={W_\Phi}^*\pi_\Phi(x)P_1P_1\sqrt{\Delta_{1\Phi}(\Psi)}V_\phi(1_\mathcal{B} \otimes \xi) \\
      &={W_\Phi}^*P_2\pi_\Phi(x)P_1\sqrt{\Delta_{1\Phi}(\Psi)}V_\phi(1_\mathcal{B} \otimes \xi) \\
      &={(P_2W_\Phi)}^*\pi_\Phi(x)\left(P_1\sqrt{\Delta_{1\Phi}(\Psi)}V_\phi (1_\mathcal{B} \otimes \xi)\right).
    \end{split}
\end{equation*}

\end{proof}

\begin{remark}
 We observe the following from Theorem \ref{unieq}:
 \begin{enumerate}
    \item  $P_1 = P_{\overline{\text{Ran}(\Delta_{1\Phi}(\Psi)^*)}}$ and $P_2 = P_{\overline{\text{Ran}(\Delta_{2\Phi}(\Psi)^*)}}$ (see \cite[Remark 3.2.5]{MJbook} for further details).

    \item One may naturally ask: ``Is it possible to discard the condition that $\text{ker}(\Delta_{1 \Phi}(\Psi))$ and $\text{ker}(\Delta_{2 \Phi}(\Psi))$ are complemented?" As of now, we do not have an affirmative answer. One approach to show that $\text{ker}(\Delta_{1 \Phi}(\Psi))$ is complemented is to show that $\text{Range}(\Delta_{1 \Phi}(\Psi))$ is closed. 
\end{enumerate}
\end{remark}

Next, we want to define a one to one correspondence between all the maps related to the completely positive map $\Psi$ and the Radon-Nikod\'ym derivative of $\Psi$ with respect to $\Phi$.

For $\Phi \in \mathcal{CP}(E, F),$ we define $\hat{\Phi} := \{\Psi \in \mathcal{CP}(E, F): \Phi \sim \Psi\}$. Let $\Psi_1, \Psi_2 \in \mathcal{CP}(E, F),$ we write $\hat{\Psi}_1 \leq \hat{\Psi}_2$ if ${\Psi}_1 \preceq {\Psi}_2.$ Next, we define
$$[0, \hat{\Phi}] := \{\hat{\Psi}: \Psi \in \mathcal{CP}(E, F), \hat{\Psi} \leq \hat{\Phi}\},$$
and 
$$[0,I]_\Phi := \{T \oplus S \in \pi_\Phi(E)': \| T \oplus S \| \leq 1\}.$$

\begin{theorem}\label{iso}
    Let $\Phi \in \mathcal{CP}(E, F).$ The map $\hat{\Psi} \mapsto \Delta_\Phi(\Psi)$ is an order-preserving isomorphism from $[0, \hat{\Phi}]$ to $[0,I]_\Phi.$
\end{theorem}

\begin{proof}
The map $\hat{\Psi} \mapsto \Delta_\Phi(\Psi)$ is well defined as seen in Theorem \ref{main}. Let $\Psi_1, \Psi_2 \in \mathcal{CP}(E, F)$ such that $\Psi_1 \preceq \Phi, \Psi_2 \preceq \Phi$ and $\Delta_\Phi(\Psi_1) = \Delta_\Phi(\Psi_2)$. Then $\Psi_1 \sim \Phi_{\sqrt{\Delta_\Phi(\Psi_1)}} = \Phi_{\sqrt{\Delta_\Phi(\Psi_2)}} \sim \Psi_2.$ So, $\hat{\Psi}_1 = \hat{\Psi}_2$, which implies that the map is injective. 

Next we show that the map is surjective.
Let $T \oplus S \in [0,I]_\Phi.$ Then by Lemma \ref{rep}, $\Phi_{\sqrt{T \oplus S}} \in \mathcal{CP}(E,F).$ Since $I - T$ is positive, it follows from the proof of Theorem \ref{phiT} that $\phi_{I-T} = \phi - \phi_T$ is completely positive. Therefore, we have $\Phi_{\sqrt{T \oplus S}} \preceq \Phi$. Moreover, by Theorem \ref{main}, there exists an operator $\Delta_\Phi(\Psi) \in \pi_\Phi(E)'$ such that $\Phi_{\sqrt{T \oplus S}} \sim \Phi_{\sqrt{\Delta_\Phi(\Psi)}}$. Since $\phi_T = \phi_{\Delta_{1\Phi}(\Phi_{\sqrt{T\oplus S}})},$ injectivity of the map $T \mapsto \phi_T,$ implies $\Delta_{1\Phi}(\Phi_{\sqrt{T\oplus S}}) = T$. Thus, by Remark \ref{nondeg} $(2)$, we have $\Delta_{\Phi}(\Phi_{\sqrt{T\oplus S}}) = T \oplus S.$
    
Let $\hat{\Psi}_1, \hat{\Psi}_2 \in [0, \hat{\Phi}]$ such that $\hat{\Psi}_1 \leq \hat{\Psi}_2$ then $\Psi_1 \preceq \Psi_2 \preceq \Phi.$ Similar calculations as seen in Lemma \ref{contraction}, imply ${J_\phi(\psi_1)}^*J_\phi(\psi_1) \leq {J_\phi(\psi_2)}^*J_\phi(\psi_2)$, that is $\Delta_{1\Phi}(\Psi_1) \leq \Delta_{1\Phi}(\Psi_2).$ By Remark \ref{nondeg} $(2)$, we get  $\Delta_{\Phi}(\Psi_1) \leq \Delta_{\Phi}(\Psi_2).$
Conversely, if, for $T_1 \oplus S_1, T_2 \oplus S_2 \in \pi_\Phi(E)'$, $0 \leq T_1\oplus S_1 \leq T_2 \oplus S_2 \leq I$, then we know that, $0 \leq T_1 \leq T_2 \leq I$, where $T_1, T_2 \in \pi_\phi(\mathcal{A})'.$ This implies that $\phi_{T_1} \leq \phi_{T_2},$ and thus we get $\Phi_{\sqrt{T_1\oplus S_2}} \preceq \Phi_{\sqrt{T_2\oplus S_2}}.$

\end{proof}

\begin{definition}
    Let $\Phi \in \mathcal{CP}(E,F).$ Then we say $\Phi$ is pure, if for any $\Psi \in \mathcal{CP}(E,F)$ with $\hat{\Psi} \leq \hat{\Phi},$ there is a $\lambda > 0$ such that $\Psi \sim \lambda \Phi.$
\end{definition}

\begin{proposition}
    Let $\Phi \in \mathcal{CP}(E,F)$ be a non-zero map. Then $\Phi$ is pure if and only if $\pi_\Phi(E)'= \mathbb{C}I$. 
\end{proposition}

\begin{proof}
    First, let $0 \neq \Phi \in \mathcal{CP}(E,F)$ be pure. Let $T \oplus S \in \pi_\Phi(E)'$ with $0 \leq T \oplus S \leq I.$ Then by Theorem \ref{iso}, $\Phi_{\sqrt{T \oplus S}} \preceq \Phi.$ 
    Since, $\Phi$ is pure, there exists a $\lambda > 0$ such that $\Phi_{\sqrt{T \oplus S}} \sim \lambda \Phi = \Phi_{\lambda I}.$ Indeed by Stinespring's construction and Lemma \ref{rep}, for $x \in E$, we have
    \begin{equation*}
        \lambda \Phi(x) = \lambda {W_\Phi}^*\pi_\Phi(x)V_\phi (1_\mathcal{B} \otimes \xi)
        =  {W_\Phi}^*\sqrt{\lambda I}\pi_\Phi(x)\sqrt{\lambda I}V_\phi (1_\mathcal{B} \otimes \xi)
        = \Phi_{\lambda I}.
    \end{equation*}
    Hence, $T \oplus S = \lambda^2 I.$
Therefore, the commutant $\pi_\Phi(E)' = \mathbb{C}I.$ 

Conversely, let $\Psi \in \mathcal{CP}(E, F)$ be such that $\hat{\Psi} \leq \hat{\Phi}$. By Theorem \ref{iso} and using the fact that $\pi_\Phi(E)' = \mathbb{C}I$, there exists $\lambda I \in \pi_\Phi(E)'$ with $\lambda > 0$ such that $\Psi \sim \Phi_{\sqrt{\lambda}I} = \sqrt{\lambda}\Phi$. Thus, $\Phi$ is pure.
\end{proof}

\begin{center}
		\textbf{Statements and Declarations}
	\end{center}
	
	\textbf{Funding}\\
	The first author's research is supported by the Prime Minister’s Research Fellowship (PMRF ID: 2001291) provided by the Ministry of Education (MoE), Govt. of India.\\

	\textbf{Data Availability}\\
	Data sharing not applicable to this article as no datasets were generated or
	analysed during the current study\\
	
	\textbf{Conflict of interest}\\
	On behalf of all authors, the corresponding author states that there is no conflict of interest.



	\bibliographystyle{amsplain}

\end{document}